\documentclass[sn-mathphys]{sn-jnl}

\jyear{2021}%

\theoremstyle{thmstyleone}%
\newtheorem{theorem}{Theorem}[section]
\newtheorem{corollary}[theorem]{Corollary}
\newtheorem{problem}[theorem]{Problem}

\theoremstyle{thmstyletwo}%
\newtheorem{example}{Example}%
\newtheorem{remark}{Remark}%

\theoremstyle{thmstylethree}%
\newtheorem{definition}{Definition}%

\usepackage{xcolor}
\usepackage[makeroom]{cancel}
\definecolor{mypink1}{rgb}{1.00,0.10,1.00}
\def\cA{{{\mathcal A}}}

\def\cC{{{\mathcal C}}}

\def\cF{{{\mathcal F}}}
\def\cG{{{\mathcal G}}}
\def\cH{{{\mathcal H}}}

\def\cK{{{\mathcal K}}}

\def\cS{{{\mathcal S}}}

\def\cQ{{{\mathcal Q}}}

\def\N{\mathbb N}

\def\R{\mathbb R}
\def\dfrac#1#2{{\displaystyle{#1\over#2}}}

\def\eps{\varepsilon}
\def\epsilon{\eps}

\def\01^N{\{0,1\}^{n}}

\def\1^K{\{0,1\}^{k}}
\def\n+1pla{(i_1,\cdots,i_{n+1})}
\def\u^n+1{\{0,1\}^{n+1}}

\def\cero{0_X}

\def\normx{\| x \|}

\begin{document}

\title[Sublinear Scalarizations for Proper and Approximate Proper Efficiency]{Sublinear Scalarizations for Proper and Approximate Proper Efficient Points in Nonconvex Vector Optimization}


\author*[ ]{\fnm{Fernando} \sur{García-Castaño}}\email{fernando.gc@ua.es}

\author[ ]{\fnm{Miguel Ángel} \sur{Melguizo-Padial}}\email{ma.mp@ua.es}
\equalcont{These authors contributed equally to this work.}

\author[ ]{\fnm{G.} \sur{Parzanese}}\email{gpzes@yahoo.it}
\equalcont{These authors contributed equally to this work.}

\affil{\orgname{$^{1}$University of Alicante},  \orgaddress{\street{Carretera de San Vicente del Raspeig s/n},  \postcode{03690}, \city{San Vicente del Raspeig - Alicante},  \country{Spain}}}

%


\abstract{We show that under a separation property, a $\cQ$-minimal point in a normed space is the minimum of a given sublinear function. This fact provides sufficient conditions, via scalarization, for nine types of proper efficient points; establishing a characterization in the particular case of Benson proper efficient points.  We also obtain necessary and sufficient conditions in terms of scalarization for approximate Benson  and Henig proper efficient points.   The separation property we handle is a variation of another known property and our scalarization results  do not require convexity or boundedness assumptions.}




\maketitle

\section{Introduction}
Proper efficient points  were introduced to eliminate efficient points exhibiting some abnormal properties. They can be described in terms of separations between the ordering cone and the considered set.  Such points have been  the object of many investigations, see for example  \cite{Benson1979,Borwein1993,Gong2005,Guerraggio1994,Ha2010,Hartley1978,Henig1982,Zalinescu2015,Zheng1997}.  In \cite{Ha2010}, the author presented the notion of $\cQ$-minimal point and showed that several types of proper efficient points can be reduced in a unified form as $\cQ$-minimal points.  The following kinds of proper efficient points were studied in \cite{Ha2010}: Henig global proper efficient points, Henig proper efficient points, super efficient points, Benson proper efficient points, Hartley proper efficient points, Hurwicz proper efficient points, and Borwein proper efficient points; the latter three types considered for the first time.  Optimality conditions for proper efficient points were obtained, and scalarization for $\cQ$-minimal points was established.  Since the scalar optimization theory is widely developed, scalarization turns out to be of great importance for the vector optimization theory \cite{Ehrgott2005,Ehrgott2005b,Jahn2004,Zalinescu2015,Luc1989,Miettinen1999,Pascoletti1984,Qiu2008,Zheng2000}. In this work, we show that under a separation property called SSP, a $\cQ$-minimal point in a normed space is the minimum of a given sublinear function.  This fact provides sufficient conditions for the proper efficient points analysed in  \cite{Ha2010} and also for tangentially Borwein proper efficient points which were not considered there. The sufficient condition becomes a characterization in case of Benson proper efficient points. We note that SSP is a variation of a separation property introduced in  \cite{Kasimbeyli2010}. On one hand, our results complement those obtained in  \cite{Ha2010} making use  of a different scalar function  and also establishing conditions for tangentially Borwein proper efficient points. In our results,  for every type of  proper efficient point, we apply the separation property to a fixed $\cQ$-dilation of the ordering cone instead of to a sequence of $\epsilon$-conic neigbhourhoods  of the ordering cone (as it is done in  \cite{Kasimbeyli2010}). This fact leads us to optimal conditions for nine types of proper efficient points (instead of two types  in  \cite{Kasimbeyli2010}) deriving scalarization results in the setting of normed spaces under weaker assumptions than those in \cite{Kasimbeyli2010} and making use of the same scalar function.  In addition, our characterization of Benson proper efficient points shed light to the last question stated in the conclusions of \cite{Guo2017}.
 
Recently, several authors have been interested in introducing and studying approximate proper efficiency notions.  The common idea in the concepts of approximate efficiency is to consider a set that approximates the ordering cone, that does not contain the origin in order to impose the approximate efficiency (or non-domination) condition.  In \cite{Gutierrez2012,Gutierrez2016}, the notions of approximate proper efficient points in the senses of Benson and Henig were introduced extending and improving the most important concepts of approximate proper efficiency given in the literature at the moment. In addition,  the authors characterized such approximate efficient points through scalarization assuming generalized convexity conditions. In this manuscript, we adapt  the approach followed to obtain optimal conditions for $\cQ$-minimal points to establish new characterizations of Benson and Henig approximate proper efficient points  through scalarizations.  Again, our results are based on SSP and we do not impose any kind of convexity assumption. So, our results complement those in \cite{Gutierrez2012,Gutierrez2016}.

The paper is organized as follows. We introduce preliminary terminology in Section \ref{sec:notation}.  In Section \ref{sec:tmas_separacion_conos}, we introduce SSP and establish two separation theorems, Theorems \ref{thm:primer_tma_separacion_C_y_K} and \ref{thm:separacion_para_aplicar_Q_minimales}. The first one provides an extension of \cite[Theorem 4.3]{Kasimbeyli2010} to normed spaces in which some assumptions for the equivalence have been relaxed, and the latter, that is our main separation result,  provides some optimal conditions for proper and approximate proper efficient points. Theorem~\ref{thm:scalarization_Q_minimals} shows that under SSP,  $\cQ$-minimal points can be obtained minimizing a sublinear function explicitly defined.  Corollary~\ref{coro:scalarization_proper_minimals} particularizes the former necessary condition for each type of proper efficient point.  Corollary~\ref{coro:Caract_Benson_Scalar_Problem} characterizes Benson proper efficient points via scalarization, and Corollaries \ref{coro:charact_scalarization_GHe} and  \ref{coro:charact_scalarization_TBo} characterize, respectively, Henig global and tangentially Borwein proper efficient points under some extra assumptions.  In Section \ref{sec:approximate_proper}, we recall the notions of approximate efficiency in the sense of Benson and of Henig. Theorems~\ref{thm:escalar_neces_approximate_Benson} and \ref{thm:escalar_sufficient_approximate_Benson} provide, respectively, necessary and sufficient conditions through scalarization for approximate Benson proper efficient points; in a similar way but under extra assumptions, we establish Corollary \ref{cor:escalar_characterization_approximate_Henig} that characterizes approximate Henig proper points. 
\section{Notation and previous definitions}\label{sec:notation}
Throughout the paper $X$ will denote a normed space, $\| \cdot \|$ the norm on $X$, $X^*$ the dual space of $X$, $\| \cdot \|_*$ the norm on $X^*$, and  $0_X$  the origin of $X$. By $B_X$ (resp. $B^{\circ}_X$) we denote the closed (resp. open) unit ball of $X$ and by $S_X$ we denote the unit sphere, i.e., $B_X:=\{x \in X \colon \| x \| \leq 1\}$, $B^{\circ}_X=\{x \in X \colon \| x \| < 1\}$, and $S_X:=\{x \in X \colon \| x \| = 1\}$.
Given a subset $S \subset X$, we denote by $\overline{S}$ (resp. bd($S$), int($S$), co($S$), $\overline{\mbox{co}}(S)$) the closure (resp. the boundary, the interior, the convex hull, the closure of the convex hull) of $S$.  Besides, for every $f\in X^*$,  we will denote by $\sup_Sf$ (resp. $\inf_Sf$) the supremum (resp. infimum) of $f$ on the set $S$. By $\mathbb{R}_+$ (resp. $\mathbb{R}_{++}$) we denote the set of non negative real numbers (resp.  strictly positive real numbers). A subset $\mathcal{C}\subset X$ is said to be a cone if $\lambda x \in \mathcal{C}$ for every $\lambda \geq 0$ and $x \in \mathcal{C}$. Let  $\mathcal{C}\subset X$ be a cone: $\mathcal{C}$ is said to be non-trivial if $\{\cero\}\subsetneq \mathcal{C}\subsetneq X$,  $\mathcal{C}$ is said to be convex if it is a convex subset of $X$, $\mathcal{C}$ is said to be pointed if $(-\mathcal{C})\cap \mathcal{C}=\{0_X\}$, and $\mathcal{C}$ is said to be solid if $\mbox{int}(\cC)\not = \emptyset$.  All cones in this manuscript are assumed to be non-trivial unless stated otherwise.  From now on, $A_0$ denotes  $A\cup \{0_X\}$ for every subset $A\subset X$.  An open cone $\cQ\subset X$ is an
 open set such that $\cQ_0$ is a (non-trivial) cone.  An open cone $\cQ$ is said to be pointed (resp. convex) if $\cQ_0$ is a pointed (resp. convex) cone on $X$. Fixed a subset $S \subset X$, we define the cone generated by $S$ as cone$(S):=\{\lambda s\colon \lambda\geq 0,\, s \in S\}$ and $\overline{\mbox{cone}}(S)$ stands for the closure of cone$(S)$. A non-empty convex subset $B$ of a convex cone $\cC$ is said to be a base for $\cC$ if $\cero \not \in \overline{B}$ and  for every $x \in \cC\setminus\{\cero\}$ there exist unique $\lambda_x > 0$, $b_x\in B$ such that $x=\lambda_x b_x$. 
Given a cone $\cC\subset X$,  its dual cone is defined by $\cC^*:=\{f \in X^*\colon f(x)\geq 0,\, \forall x \in \cC\}$ and the set of strictly positive functionals by $\cC^{\#}:=\{f \in X^*\colon f(x)> 0,\, \forall x \in \cC,x\not = \cero\}$.  
In general, $\mbox{int}(\cC^*)\subset \cC^{\#}$.  It is known that a convex cone $\cC\subset X$ has a base if and only if $\cC^{\#}\not = \emptyset$, and $\cC^{\#}\not = \emptyset$ implies that  $\cC$ is pointed. In particular, for every $f \in \cC^{\#}$, the set $B:=\{x \in \cC\colon f(x)=1\}$ is a base for $\cC$. A convex cone $\cC$ is said to have a bounded base if there exists a base $B$ for $\cC$ such that it is a bounded subset of $X$. It is known that $\cC$ has a bounded base if and only if $\mbox{int}(\cC^*)\not = \emptyset$ if and only if $\cero$ is a denting point for $\cC$ (see \cite[Theorem 3.8.4]{Jameson1970} for the first equivalence and \cite{GARCIACASTANO2018,GARCIACASTANO20151178,GARCIACASTANO2021} for further information about dentability and optimization).

A convex cone $\cC$ is said to have a (weak) compact base if there exists a base $B$ of $\cC$ which is a (weak) compact subset of $X$.   A pointed cone admits a compact base if and only if it is locally compact  if and only if the cone satisfies the strong property $(\pi)$  if and only if  there exists $f\in \cC^{\#}$ such that $\cC\cap \{f \leq \lambda\}$ is compact, $\forall \lambda>0$.  We refer the reader to \cite[p. 338]{Kothe1983} for the first equivalence,  to \cite[Definition 2.1]{Han1994} for the definition of strong property $(\pi)$ and to \cite[Remark 2.1]{Qiu2001} for the last two equivalences.  The following two sets are called augmented dual cones of a given cone $\cC$ and they were introduced in \cite{Kasimbeyli2010}, $\cC^{a*}:=\{(f,\alpha)\in \cC^{\#}\times \R_+ \colon f(x)-\alpha\| x \| \geq 0,\, \forall x \in \cC\}$ and $\cC^{a\#}:=\{(f,\alpha)\in \cC^{\#}\times \R_+ \colon f(x)-\alpha\| x \| > 0,\, \forall x \in \cC,x \not = \cero\}$. Clearly $\cC^{a\#}\subset \cC^{a*}$. Now, we introduce the following augmented dual cones $\cC^{a*}_+:=\{(f,\alpha)\in \cC^{\#}\times \R_{++} \colon f(x)-\alpha\| x \| \geq 0,\, \forall x \in \cC\}$, $\cC^{a\#}_+:=\{(f,\alpha)\in \cC^{\#}\times \R_{++} \colon f(x)-\alpha\| x \| > 0,\, \forall x \in \cC,x \not = \cero\}$. It is clear that $\cC^{a\#}_+\subset \cC^{a*}_+$. 

Let $\cC\subset X$ be a pointed convex cone, then $\cC$ provides a partial order on $X$, say $\leq$, in the following way, $x \leq y \Leftrightarrow y-x\in \cC$. In this situation, we say that $X$ is a partially ordered normed space and $\cC$ is the ordering cone. Let $X$ be a partially ordered normed space, $\cC\subset X$ the ordering cone, and $A\subset X$ a subset. We say that $x_0\in A$ is an efficient (or Pareto minimal) point of $A$, written $x_0 \in \mbox{Min}(A,\cC)$, if $(x_0-\cC)\cap A=\{x_0\}$. Next, we define some types of proper efficient points.  Note that (i)-(iii) and (v)-(viii) are taken from  \cite[Definition 21.3]{Ha2010} (see also \cite[Definition 2.4.4]{Zalinescu2015}), (iv) is taken from \cite{Makarov1999} adapting from maximal to minimal proper efficient point,   and (ix) is taken  from \cite{Eichfelder2014}.  The latter was obtained adapting \cite[Definition~2]{Borwein1977} and was called Borwein proper efficient point, but we have changed the name to distinguish (v) and (ix) below. Recall that fixed a subset $A\subset X$ and a point $\bar{x} \in \overline{A}$, the contingent cone to $A$ at $\bar{x}$ is defined by  
$T(A,\bar{x}):=\{\lim_n\lambda_n(x_n-\bar{x}) \in X\colon (\lambda_n)\subset \R_{++} \mbox{, } (x_n)\subset A \mbox{,  and }\lim_nx_n=\bar{x}\}$,
see \cite{aubin1990set} for details.

\begin{definition} \label{defi:proper_minimal_point} 
Let $X$ be a partially ordered normed space, $\cC$ the ordering cone, and $A \subset X$ a subset.
\begin{itemize}
\item[(i)]  $x_0 \in A$ is called a positive proper efficient point of $A$, $x_0 \in \mbox{Pos}(A,\cC)$, if there exists $f\in \cC^{\#}$ such that $f(x_0)=\inf_A f$.
\item[(ii)] $x_0 \in A$ is called a Hurwicz proper efficient point of $A$, $x_0 \in \mbox{Hu}(A,\cC)$, if $\overline{\mbox{co}}(\cK) \cap (-\cC) =\{\cero\}$ for $\cK=\mbox{cone}((A-x_0) \cup \cC)$.
\item[(iii)] $x_0 \in A$ is called a Benson proper efficient point of $A$, $x_0 \in \mbox{Be}(A,\cC)$, if  $\cero\in \mbox{Min}(\overline{\mbox{cone}}(A+\cC-x_0),\cC)$. 
\item[(iv)] $x_0 \in A$ is called a Hartley proper efficient point of $A$, $x_0 \in \mbox{Ha}(A,\cC)$, if $x_0 \in \mbox{Min}(A,\cC)$ and there exists $M > 0$ such that if  $f \in \cC^*$ and $f(x-x_0)<0 $ for some $x \in A$, then there exists $g \in \cC^*$, $g\not = 0$,  satisfying $\Vert g\Vert f(x-x_0)  \geq -\Vert f \Vert Mg(x-x_0)$.  
\item[(v)]  $x_0 \in A$ is called a Borwein proper efficient point of $A$, $x_0 \in \mbox{Bo}(A,\cC)$, if $\overline{\mbox{cone}}(A-x_0)\cap (-\cC) =\{\cero\}$. 
\item[(vi)] $x_0 \in A$ is called a Henig global proper efficient point of $A$, $x_0 \in \mbox{GHe}(A,\cC)$, if $x_0 \in \mbox{Min}(A,\cK)$ for some convex cone $\cK$ such that $\cC\setminus \{\cero\}\subset \mbox{int}(\cK)$. 
\item[(vii)] $x_0 \in A$ is called a Henig  proper efficient point of $A$, $x_0 \in \mbox{He}(A,\cC)$, if $\cC$ has a base $B$ and there exists $\epsilon>0$ such that $\overline{\mbox{cone}}(A-x_0)\cap (-\overline{\mbox{cone}}(B+\epsilon B_X))=\{\cero\}$.
\item[(viii)] $x_0 \in A$ is called a super efficient point of $A$, $x_0 \in \mbox{SE}(A,\cC)$, if there exists $\rho >0$ such that $\overline{\mbox{cone}}(A-x_0)\cap (B_X-\cC) \subset \rho B_X$.
\item[(ix)] $x_0\in A$ is a tangentially Borwein proper efficient point of $A$, written $x_0 \in \mbox{TBo}(A,\cC)$, if $T(A+\cC,x_0)\cap(-\cC)=\{\cero\}$.
\end{itemize}
\end{definition}
Proper efficient points were introduced for two main reasons: first, to eliminate certain anomalous minimal points; second, to establish some equivalent scalar problems whose solutions provide at least most of the minimal points. We have the following chain of inclusions (see \cite[Proposition 21.4]{Ha2010} and \cite[Proposition 2.4.6]{Zalinescu2015}).
$\mbox{Pos}(A,\cC)\subset \mbox{Hu}(A,\cC)\subset \mbox{Be}(A,\cC)$, $\mbox{Pos}(A,\cC)\subset \mbox{GHe}(A,\cC)$, $\mbox{SE}(A,\cC)\subset \mbox{Ha}(A,\cC)\subset \mbox{Be}(A,\cC)\subset \mbox{Bo}(A,\cC)$, $\mbox{SE}(A,\cC)\subset\mbox{GHe}(A,\cC)\subset \mbox{Be}(A,\cC)$, $\mbox{SE}(A,\cC)\subset \mbox{He}(A,\cC)$.  In addition, by \cite[Theorem 3.44]{Jahn2004}, if we have $T(A+\cC,x_0)\subset \overline{\mbox{cone}}(A+\cC,x_0)$, then $\mbox{Be}(A,\cC)\subset \mbox{TBo}(A,\cC)$. Furthermore, under extra assumptions on $A$ or $\cC$,  we can find more inclusions (again \cite{Ha2010,Zalinescu2015}). 

\section{Separation theorems and scalarization for Q-minimal points}\label{sec:tmas_separacion_conos}
We begin this section introducing a separation property of cones called strict separation property (SSP for short).  Later, we establish two  theorems which separate two cones in a normed space by a sublevel set of a sublinear function. The second separation theorem will be key to determining optimality conditions for the proper efficient points introduced in section \ref{sec:notation}.

Let us recall the following. Throughout the paper, we denote by $A_0$ the set $A\cup \{\cero\}$, for any $A\subset X$. On the other hand,  every cone $\cC\subset X$ we consider is assumed non-trivial, i.e.,  $\{\cero\}\subsetneq \mathcal{C}\subsetneq X$.  For any cone $\cC$, we define the following convex sets $\cC_{\scriptscriptstyle \wedge}  :={\mbox{co}}(\cC\cap S_X)$ and $\cC_{\scriptscriptstyle \vee}:={\mbox{co}}((\mbox{bd}(\cC)\cap S_X)_0)$ to be used in the following separation property.
\begin{definition}\label{def:estrict_separation}
Let $X$ be a normed space and $\cC$, $\cK$ cones on $X$.  We say that the pair of cones $(\cC, \,\cK)$ has the strict separation property (SSP for short) if $\cero\not \in \overline{\cC_{\scriptscriptstyle \wedge}-\cK_{\scriptscriptstyle \vee}}$.
\end{definition}

%

\begin{remark}\label{remSSPcomplementario}
Given two cones $ \cC,\cK \subset X$ we have the following.
\begin{itemize}
\item[(i)]  $(\cC,\cK)$ has SSP $\Leftrightarrow (\overline{\cC},\cK)$ has SSP  $\Leftrightarrow (\cC,\overline{\cK})$ has SSP $\Leftrightarrow (\overline{\cC},\overline{\cK})$ has SSP.
\item[(ii)] Since $\mbox{bd}(A)=\mbox{bd}(X \setminus A)$ for every subset $A\subset X$ and  $\mbox{bd}(\cK)=\mbox{bd}(\cK\setminus \{\cero\})$ for every non-trivial cone $\cK\subset X$, it follows that $(\cC, \,\cK)$ has SSP if and only if  $(-\cC, \,-\cK)$ has SSP if and only if $(\cC, \,(X \setminus \cK)_0)$ has SSP.
\end{itemize}

\end{remark}
In \cite[Definition 4.1]{Kasimbeyli2010} the author introduces the separation property SP for closed cones under the condition $\cero \not \in \overline{\cC_{\scriptscriptstyle \wedge}}-\overline{\cK_{\scriptscriptstyle \vee}}$. It is clear that, for closed cones, SSP implies SP. The reverse  is true in reflexive Banach spaces because on such spaces the closure of the  difference of two bounded convex sets equals the difference of the closures of the sets (referring to the Minkowski difference).

The following result provides a version of \cite[Theorem~4.3]{Kasimbeyli2010} in the setting of general normed spaces. It is worth pointing out that we do not need extra assumptions  to obtain an equivalence (such as the cones to be convex and closed or the space to be of finite dimension). 

\begin{theorem}\label{thm:primer_tma_separacion_C_y_K}
Let $X$ be a normed space and $\cC$, $\cK$ cones on $X$.  The following assertions are equivalent.
\begin{itemize}
\item[(i)] $(\cC, \,\cK)$ has SSP.
\item[(ii)] There exist $\delta_2>\delta_1 > 0$ and $f \in X^*$ such that $(f,\alpha)\in \cC^{a\#}_+$ and $0< f(y)+\alpha \| y \|$  for every $\alpha \in (\delta_1,\delta_2)$  and $y \in \mbox{bd}(-\cK)$, $y\not =\cero$.
\item[(iii)] There exist $\delta_2>\delta_1 > 0$ and $f \in X^*$ such that $(f,\alpha)\in \cC^{a\#}_+$ and $f(x)+\alpha \normx<0< f(y)+\alpha \| y \|$ for every $\alpha \in (\delta_1,\delta_2)$,  $x \in -\overline{co}(\cC )$,  $x\not =\cero$, and $y \in \mbox{bd}(\textcolor{black}{-\cK})$, $y\not =\cero$.
\end{itemize}
\end{theorem}
\begin{proof}
(i)$\Rightarrow$(ii) Since $\cero\not \in \overline{{\cC}_{\scriptscriptstyle \wedge}-{\cK_{\scriptscriptstyle \vee}}}$, by \cite[Theorem 2.12]{Montensinos2ndbook}, there exist $f \in X^* $ and $\beta_1, \beta_2 \in \mathbb{R}$ such that $f(0)=0 < \beta_1 < \beta_2 <f(c-k)=f(c)-f(k)$ for every $c\in {\cC}_{\scriptscriptstyle \wedge}$, $k\in \cK_{\scriptscriptstyle \vee}$. 
Then $f(k) < \beta_1 + f(k) < \beta_2 + f(k)<f(c)$,  for every $c\in {\cC}_{\scriptscriptstyle \wedge}$, $k\in\cK_{\scriptscriptstyle \vee}$. 
Denote $S:=  \sup_{\cK_{\scriptscriptstyle \vee}} f  $. Then $0\leq S  < + \infty$ because $\cero \in \cK_{\scriptscriptstyle \vee}$ and the set $\cK_{\scriptscriptstyle \vee}$ is bounded. 
Furthermore,  for every $c\in {\cC}_{\scriptscriptstyle \wedge}$, $k\in \cK_{\scriptscriptstyle \vee}$ we have $f(k)\leq S< \beta_1 +S < \beta_2 + S\leq f(c)$. Therefore, denoting $\delta_1 :=\beta_1 +S$ and $\delta_2 :=\beta_2 +S$ we get that $ f(k)< \alpha< f(c)$ for every $c\in {\cC}_{\scriptscriptstyle \wedge}$, $k\in \cK_{\scriptscriptstyle \vee}$ and $\alpha \in (\delta_1, \delta_2)$.  On the other hand, since $\cero \in \cK_{\scriptscriptstyle \vee}$, it follows that $f(c)>0$ for every $c\in {\cC}_{\scriptscriptstyle \wedge}$. Thus $f\in C^{\#}$. 
Now, fix arbitrary $\alpha \in (\delta_1, \delta_2) $ and $y \in \mbox{bd}(\cK) \cap S_X $. As $f(y)<\alpha $ and  $\Vert y \Vert=1$, we get that  $f(y)-\alpha\Vert y \Vert<0 $, which leads to  $f(-y)+\alpha\Vert y \Vert>0$. 
Therefore  $f(y)+\alpha\Vert y \Vert>0$ for every  $y \in -\mbox{bd}(\cK)\setminus \{\cero \}$.  Fix now  $\alpha \in (\delta_1, \delta_2)$ and $x \in (-\cC)\cap S_X $. Again, $\cero \in \cK_{\scriptscriptstyle \vee}$ implies $f(\cero)=0<\alpha <f(-x) $. Since $\Vert -x \Vert=1$, we get that  $0<f(-x)-\alpha\Vert -x \Vert $, which leads to  $f(x)+\alpha\Vert x \Vert<0$. Consequently, $f(x)+\alpha\Vert x \Vert<0$ for every   $x \in -\cC$, $x\not =\cero$. Clearly $(f,\alpha)\in \cC^{a\#}_+$.

(ii)$\Rightarrow$(iii) Consider $(f,\alpha)\in \cC^{a\#}_+$ for some $\alpha \in (\delta_1,\delta_2)$.  Then $f(x)+\alpha \normx<0$ for every $x \in -\cC$,  $x\not =\cero$.  Now,  sublinearity of $f(\cdot)+\alpha \| \cdot \|$ yields  $f(x)+\alpha \| x \|< 0$ for every $x \in -\mbox{co}(\cC )$,  $x\not =\cero$.  Next, we will prove the precedent inequality for  $x \in -\overline{\mbox{co}}(\cC )$, $x\not = \cero$.  Assume, contrary to our claim, that there exist some   $\bar{x} \in - (\overline{\mbox{co}}(\cC )\setminus \mbox{co}(\cC))$, $\bar{x}\not =\cero$, such that $f(\bar{x})+\alpha \| \bar{x} \|= 0$. Then, there exists a sequence $\{x'_n\} \subset -\mbox{co}(\cC) $ such that $\underset{n\rightarrow \infty}{\lim}x'_n=\bar{x}$. Now, we fix some $\hat{\alpha} \in (\delta_1,\delta_2)$ such that $\alpha<\hat{\alpha}$. 
Then,
$\underset{n\rightarrow \infty}{\lim}f(x'_n)+ \hat{\alpha} \|x'_n\|= \underset{n\rightarrow \infty}{\lim}  f(x'_n)+ \alpha \|x'_n\| +  \underset{n\rightarrow \infty}{\lim} (\hat{\alpha}-\alpha) \| x'_n \|= (\hat{\alpha}-\alpha) \| \bar{x}\| >0$,
because $\bar{x} \neq \cero$. Hence, there exists some $n_0 \in \mathbb{N}$ such that $f(x'_{n_0})+ \alpha \|x'_{n_0}\|>0$,  which is impossible because  $x'_{n_0}\in - \mbox{co}(\cC)$, $x'_{n_0}\not =\cero$.

(iii)$\Rightarrow$(i) 
For any $x\in (-\cC)\cap S_X$, $y \in (-\mbox{bd}(\cK))\cap S_X$, $y\not =\cero$, and $\alpha \in (\delta_1,\delta_2)$ we have $f(x)<-\alpha< f(y)$. The former inequalities also hold for $x \in \mbox{co}((-\cC)\cap S_X)$ and $y \in \mbox{co}(-(\mbox{bd}(\cK)\cap S_X)_0)$.   Indeed, fix $\alpha \in (\delta_1,\delta_2)$ and $x \in  -\cC_{\scriptscriptstyle \wedge}=\mbox{co}((-\cC)\cap S_X)$. 
 Then  $x=\sum_{i=1}^{n}\lambda_i  x_i $ for $\lambda_i \geq 0$, $\sum_{i=1}^{n}\lambda_i=1$, $x_i\in (-\cC)\cap S_X$, and $f(x_i)<-\alpha$. Therefore, $f(x)=\sum_{i=1}^{n}\lambda_i  f(x_i) <-\alpha$.  For the same reason, $ -\alpha <f(y)$ for every $y \in -\cK_{\scriptscriptstyle \vee}=\mbox{co}(-(\mbox{bd}(\cK)\cap S_X)_0)$. This implies that $f(y)-f(x)>\frac{\delta_2-\delta_1}{2}$ for every $x \in -\cC_{\scriptscriptstyle \wedge}$ and $y \in -\cK_{\scriptscriptstyle \vee}$.  To obtain a contradiction,  suppose that $(\cC,\cK)$ does not have SSP. Then $(-\cC,-\cK)$ does not have SSP either, i.e.,  $\cero  \in \overline{-\cC_{\scriptscriptstyle \wedge}-(-\cK_{\scriptscriptstyle \vee})}$.  Then, there exist two sequences $(x_n)_n \subset -\cC_{\scriptscriptstyle \wedge}$ and $(y_n)_n \subset -\cK_{\scriptscriptstyle \vee}$ such that $0=\lim_n \| x_n-y_n \| $.  
Then, by continuity of $f$, we have $\lim_n f(y_n-x_n)=0$. 
As a consequence, there exists $n_0 \in \mathbb{N}$ such that $\mid  f(y_n)-f(x_n) \mid  < \frac{\delta_2-\delta_1}{2} $ for every $n > n_0$,  a contradiction.

\end{proof}

Next, a direct consequence of the precedent result.
\begin{corollary}\label{coro_clcoC}
Let $X$ be a normed space and $\cC$, $\cK$ cones on $X$. Then $(\cC, \cK)$ has SSP if and only if $(\overline{co}(\cC), \cK)$ has SSP.
\end{corollary}

Theorem \ref{thm:primer_tma_separacion_C_y_K} establishes that the results in \cite{Kasimbeyli2010} obtained for reflexive Banach spaces can be extended to general normed spaces under SSP. On the other hand, the following result is our main separation theorem and it will provide optimal conditions for the proper minimal points introduced in Section \ref{sec:notation} and for the approximate proper efficient points in Section \ref{sec:approximate_proper}.

\begin{theorem} \label{thm:separacion_para_aplicar_Q_minimales}
Let $X$ be a normed space and $\cC$, $\cK$ cones on $X$ such that  $\overline{co}(\cC) \cap \cK \not =\{\cero \}$. If $(\cC,\cK)$ has SSP, then $\mbox{int}(\cK) \neq \emptyset$ and there exist $\delta_2>\delta_1>0$ and $f \in X^*$ such that 
 $(f,\alpha) \in \cC^{a\#}_+$ and  $f(x)+ \alpha\Vert x \Vert <0<f(y)+\alpha \Vert y \Vert$, for every $\alpha \in (\delta_1,\delta_2)$, $ x\in -\overline{co}(\cC)$, $x\not =\cero$,  and $y\in X \setminus \mbox{int}(-\cK)$, $y\not =\cero$.
\end{theorem}

\begin{proof}
Assume that $(\cC, \,\cK)$ has SSP. By Theorem~\ref{thm:primer_tma_separacion_C_y_K}, there exist $\delta_2>\delta_1>0$ and $f \in X^*$ such that  $(f,\alpha) \in \cC^{a\#}_+$ and $f(x)+ \alpha\Vert x \Vert <0<f(y)+\alpha \Vert y \Vert$, for every $\alpha \in (\delta_1,\delta_2)$, $x\in -\overline{\mbox{co}}(\cC)$, $x\not = \cero$,  and $y\in \mbox{bd}(-\cK)$, $y\not = \cero $. First, we check that $\mbox{int}(\cK) \neq \emptyset$.  Fix an arbitrary $\bar{x}\in (-{\overline{\mbox{co}}(\cC)}) \cap (-\cK)$ such that $\bar{x}\not = \cero$. Then $f(\bar{x})+\alpha \Vert \bar{x} \Vert<0$. On the other hand, since $ \bar{x} \in \mbox{bd}(-\cK)$ implies $f(\bar{x})+\alpha \Vert \bar{x} \Vert>0$, we have $ \bar{x} \not\in \mbox{bd}(-\cK) $, so $\bar{x} \in \mbox{int}(-\cK)$. Thus $\mbox{int}(\cK) \neq \emptyset$. Next, we will prove that $ 0< f(y)+ \alpha \Vert y \Vert$ for every $y \in X \setminus \mbox{int}(-\cK)$, $y\not = \cero$. Assume, contrary to our claim, that there exists $\bar{y} \in  X \setminus \mbox{int}(-\cK)$ such that $\bar{y}\not =\cero$ and satisfying $f(\bar{y})+ \alpha \Vert \bar{y} \Vert \leq 0 $. It is clear that $\bar{y} \not \in \mbox{bd}(-\cK)$. Hence $\bar{y} \in X \setminus (-\overline{\cK})$. We consider again the former $\bar{x} \in \mbox{int}(-\cK)$ and there exists $\lambda_0 \in  (0,1)$ such that $x_0=\lambda_0 \bar{x}+ (1-\lambda_0) \bar{y}\in \mbox{bd}(-\cK)$.
As a consequence, $f(x_0)+\alpha \| x_0 \| = f( \lambda_0 \bar{x}+ (1-\lambda_0) \bar{y}) + \alpha \Vert \lambda_0 \bar{x}+ (1-\lambda_0) \bar{y} \Vert \leq
\lambda_0 (f(\bar{x})+ \alpha \Vert \bar{x} \Vert) + (1- \lambda_0)(f(\bar{y})+ \alpha \Vert \bar{y} \Vert)<0$, which contradicts $x_0 \in \mbox{bd}(-\cK)$.
\end{proof}


The following result shows the relative position of a pair of cones  having SSP.

\begin{corollary}\label{coro:pos_relativa_conos_SSP}
Let $X$ be a normed space and $\cC$, $\cK$ cones on $X$. If $\cC$ and $\cK$ have SSP,  then either $\overline{co}(\cC) \setminus \{\cero\} \subset \mbox{int}(\cK)$ or $\overline{co}(\cC) \setminus \{\cero\} \subset \mbox{int}(X \setminus \cK)$.
\end{corollary}
\begin{proof}
If $(\cC,\cK)$ has SSP, by Corollary \ref{coro_clcoC} we get that $(\overline{\mbox{co}}(\cC),\cK)$ has SSP,  and  then $(-\overline{\mbox{co}}(\cC),-\cK)$ has SSP, too. By Theorem \ref{thm:primer_tma_separacion_C_y_K},  we have $(-\overline{\mbox{co}}(\cC) \setminus \{ \cero \}) \cap \mbox{bd}(-\cK)=\emptyset$,  hence $(\overline{\mbox{co}}(\cC) \setminus \{ \cero \}) \cap \mbox{bd}(\cK)=\emptyset$.  Now, suppose that the assertion of the corollary is false. Then, there exist  $c_1$, $c_2 \in \overline{\mbox{co}}(\cC) \setminus \{\cero \}$ such that  $c_1 \in \mbox{int}( \cK) $  and   $c_2 \in \mbox{int}(X \setminus \cK)$. As $c_1 \in \overline{\mbox{co}}(\cC) \cap \cK$, Theorem \ref{thm:separacion_para_aplicar_Q_minimales} applies. 
 Hence there exists $(f,\alpha) \in \cC_+^{a\#} $ such that $f(x)+ \alpha\Vert x \Vert <0<f(y)+\alpha \Vert y \Vert $, for every $ x\in -\overline{\mbox{co}}(\cC)$, $x\not =\cero$, and  $ y\in X \setminus \mbox{int}(-\cK)$, $y\not =\cero$. Now, on one hand,  $f(-c_2)+ \alpha\Vert -c_2 \Vert <0$ because $c_2 \in \overline{\mbox{co}}(\cC)$.  But on the other hand, since $c_2 \in \mbox{int}(X \setminus \cK)\subset X \setminus \cK \subset X \setminus \mbox{int}(\cK)$, it follows that $0<f(-c_2)+ \alpha\Vert -c_2 \Vert $, a contradiction.
\end{proof}

\color{black}
We next introduce the notion of $\cQ$-minimal point (from \cite{Ha2010,Zalinescu2015}) to derive   scalarizations  for  proper efficient points in a unified way. 

\begin{definition}\label{defi:Q_minimal_point}
Let $X$ be a normed space, $\cQ \subset  X$ an open cone, and $A \subset X$ a subset. We say that $x_0\in A$ is a  $\cQ$-minimal point of $A$, $x_0\in \cQ\mbox{Min}(A)$, if $(A-x_0 ) \cap (-\cQ) =\emptyset$.
\end{definition}
The following notion is directly related to $\cQ$-minimal points.
\begin{definition}\label{defi:dilation_cone}
Let $X$ be a partially ordered normed space, $\cC\subset X$ the ordering cone, and $\cQ\subset X$ an open cone. We say that $\cQ$ dilates $\cC$ (or $\cQ$ is a dilation of $\cC$) if $\cC\setminus\{\cero\}\subset \cQ$.
\end{definition}
The following result shows that the proper efficient points introduced in Definition~\ref{defi:proper_minimal_point} are $\cQ$-minimal points with $\cQ$ being appropiately chosen cones.  Assertions (i)-(viii)  below are Assertions (iii)-(x) from \cite[Theorem 21.7]{Ha2010}, respectively, under minor adaptations (see also \cite[Theorem~2.4.11]{Zalinescu2015}). On the other hand, assertion (ix) below can be proved in a similar way as (ii) and (iii) in \cite[Theorem 21.7]{Ha2010}, so we omit the proof.  We need more terminology. Let $X$ be a partially ordered normed space,  $\cC\subset X$ the ordering cone,  and $B$ a base of $\cC$. Let $\delta_B:=\inf\{\Vert b \Vert \colon b \in B\}>0$, for every $0<\eta<\delta_B$ we define a convex, pointed and open cone $V_{\eta}(B):=\mbox{cone}(B+\eta B^{\circ}_X)$. On the other hand, given $0<\epsilon<1$ we define another open cone $\cC(\epsilon):=\{x \in X \colon d(x,\cC)<\epsilon d(x,-\cC)\}$. 
 
 
\begin{theorem} \label{thm:Idenfi_proper_con_QMIn} 
Let $X$ be a partially ordered normed space, $\cC$ the ordering cone, $A\subset X$ a subset, and $x_0 \in A$. The following statements hold.
\begin{itemize}
\item[(i)] $x_0 \in  \mbox{Pos}(A,\cC)$ if and only if  $x_0\in \cQ\mbox{Min}(A)$ for $\cQ=\{x \in X \colon f(x)>0\}$ and some $f \in \cC^{\#}$.
\item[(ii)] $x_0 \in  \mbox{Hu}(A,\cC)$ if and only if  $x_0\in \cQ\mbox{Min}(A) $ for $\cQ= -X \setminus \overline{\mbox{co}}(\cK)$  and $\cK:=\mbox{cone}((A-x_0) \cup \cC)$.
\item[(iii)] $x_0 \in  \mbox{Be}(A,\cC)$ if and only if $x_0 \in \cQ\mbox{Min}(A) $ for $\cQ= -X \setminus \overline{\mbox{cone}}(A-x_0+ \cC)$.
\item[(iv)] $x_0 \in  \mbox{Ha}(\cA,\cC)$ if and only if $x_0\in Q\mbox{Min}(A) $ for $\cQ=\cC(\epsilon)$ and $\epsilon >0$.
\item[(v)] $x_0 \in  \mbox{Bo}(\cA,\cC)$ if and only if $x_0\in \cQ\mbox{Min}(A) $ for $Q$ an open cone dilating $\cC$.
\item[(vi)] $x_0  \in  \mbox{GHe}(\cA,\cC)$ if and only if $x_0\in \cQ\mbox{Min}(A)$ for $Q$ a  pointed convex open cone dilating $\cC$.
\item[(vii)] $x_0  \in  \mbox{He}(\cA,\cC)$ if and only if $x_0\in \cQ\mbox{Min}(A) $ for $Q=V_{\eta}(B)$,  $0<\eta<\delta_B$, and a base $B$ of $\cC$.
\item[(viii)] $x_0  \in  \mbox{SE}(\cA,\cC)$ if and only if $x_0\in \cQ\mbox{Min}(A)$ for $Q=V_{\eta}(B)$, $0<\eta<\delta_B$, and $B$ a bounded base of $\cC$.
\item[(ix)] $x_0 \in  \mbox{TBo}(A,\cC)$ if and only if $x_0 \in \cQ\mbox{Min}(A) $ for $\cQ= -X \setminus T(A+ \cC,x_0)$.
\end{itemize}
\end{theorem}
In \cite{Ha2010}, the author provides some necessary and sufficient conditions for a $\cQ$-minimal point to be a solution for a scalar optimization problem.  In the following result, we provide a necessary condition  in terms of SSP for $\cQ$-minimal points to be a solution of another scalar problem.
\begin{theorem} \label{thm:scalarization_Q_minimals}
Let $X$ be a partially ordered normed space, $\cC$ the ordering cone,  $\cQ\subset X$ an open cone,  and $x_0 \in A \subset X$.  Assume that $x_0\in \cQ\mbox{Min}(A)$ and $\cC\cap \cQ\not = \emptyset$.  
If  $(\cC, \,\cQ_0)$ has SSP,  then there exists $(f,\alpha) \in \cC^{a\#}_+$ such that \[\underset{x \in A}{min} \ f(x-x_0)+ \alpha\Vert x-x_0 \Vert\] is attained only at $x_0$.
\end{theorem}
\begin{proof}
It is not restrictive to assume that $x_0=0_X$ (by translation we obtain the general case). Since $(\cC,\,\cQ_0)$ has SSP, Theorem \ref{thm:separacion_para_aplicar_Q_minimales} applies and there exists $(f,\alpha) \in \cC^{a\#}_+$ such that $f(c)+ \alpha\Vert c \Vert <0<f(x)+\alpha \Vert x \Vert$ for every $ c\in -\cC \setminus \{\cero \}$ and  $ x\in X \setminus (-\cQ)$.  Since $A \subset X \setminus (-\cQ)$, it follows that  $0<f(x)+\alpha \Vert x \Vert$ for every $x \in A $, $x \neq 0_X$.
\end{proof}
In the following result, we establish a particular version of the former one for each type of proper efficient point introduced in Definition \ref{defi:proper_minimal_point}.
\begin{corollary} \label{coro:scalarization_proper_minimals}
Let $X$ be a partially ordered normed space, $\cC$ the ordering cone,  and $x_0 \in A\subset X$.  Assume that either of the following statements holds.
\begin{itemize}
\item[(i)]  $x_0 \in  \mbox{Pos}(A,\cC)$ and take $g \in \cC^{\#}$ such that $x_0\in \cQ\mbox{Min}(A,\cC)$ for $\cQ=\{x\in X \colon g(x)>0\}$.  
\item[(ii)]  $x_0 \in   \mbox{Hu}(A,\cC)$ and take $\cQ= -X \setminus \overline{\mbox{co}}(\cK)$  for $\cK:=\mbox{cone}((A-x_0) \cup \cC)$. 
\item[(iii)] $x_0 \in  \mbox{Be}(A,\cC)$ and take $\cQ= -X \setminus \overline{\mbox{cone}}(A-x_0+ \cC)$.
\item[(iv)]  $x_0 \in  \mbox{Ha}(\cA,\cC)$ and take $\cQ=C(\epsilon)$ for some $\epsilon >0$.
\item[(v)] $x_0 \in  \mbox{Bo}(\cA,\cC)$ and take  $Q$ an open cone dilating $\cC$.
\item[(vi)] $x_0  \in  \mbox{GHe}(\cA,\cC)$ and take $Q$ a  pointed convex open cone dilating $\cC$.
\item[(vii)] $x_0  \in  \mbox{He}(\cA,\cC)$ and take $Q=V_{\eta}(B)$,  $0<\eta<\delta_B$,  for some base $B$ of $\cC$.
\item[(viii)] $x_0  \in  \mbox{SE}(\cA,\cC)$ and take $Q=V_{\eta}(B)$, $0<\eta<\delta_B$,  for some bounded base $B$ of $\cC$.
\item[(ix)] $x_0 \in  \mbox{TBo}(A,\cC)$ and take $\cQ= -X \setminus T(A+ \cC,x_0)$.
\end{itemize}

If  $(\cC, \,\cQ_0)$ has SSP,  then there exists $(f,\alpha) \in \cC^{a\#}_+$ such that 
\begin{equation}\label{eq:scalarization_Q_proper_minimals}
\underset{x \in A}{min} \ f(x-x_0)+ \alpha\Vert x-x_0 \Vert
\end{equation}
 is attained only at $x_0$.
\end{corollary}
\begin{proof}
We begin noting the inclusions $\cC\setminus \{\cero\}\subset \cQ$ for (i)-(ix).  The inclusion for the cases (i) and (iv)-(ix) is trivial.  For case (ii), we apply the following chain of equivalences: $\overline{\mbox{co}}(\cK)\cap (-\cC)=\{\cero\}\Leftrightarrow\overline{\mbox{co}}(\cK)\cap (-\cC\setminus \{\cero\})=\emptyset\Leftrightarrow -\cC\setminus \{\cero\}\subset X\setminus \overline{\mbox{co}}(\cK)=-\cQ$. 
For case (iii) we apply $\overline{\cK}\cap (-\cC)=\{\cero\}\Leftrightarrow\overline{\cK}\cap (-\cC\setminus \{\cero\})=\emptyset\Leftrightarrow -\cC\setminus \{\cero\}\subset X\setminus\overline{\cK}=\cQ.$

Now assume that either of (i)-(ix) holds.  By Theorem~\ref{thm:Idenfi_proper_con_QMIn},  $x_0\in \cQ\mbox{Min}(A,\cC)$ and by the former paragraph,  $\cC\cap \cQ\not=\emptyset$.  Then Theorem \ref{thm:scalarization_Q_minimals} applies.
\end{proof}

The following result characterizes Benson proper minimal points via (\ref{eq:scalarization_Q_proper_minimals}).

\begin{corollary}\label{coro:Caract_Benson_Scalar_Problem}
Let $X$ be a partially ordered normed space, $\cC$ the ordering cone,  and $x_0 \in A\subset X$. Assume that $(-\cC, \, \overline{\mbox{cone}}(A-x_0+ \cC))$ has SSP. Then $x_0 \in  \mbox{Be}(A,\cC)$ if and only if there exists $(f,\alpha) \in \cC^{a\#}_+$ such that $\underset{x \in A}{min} \ f(x-x_0)+ \alpha\Vert x-x_0 \Vert$ is attained only at $x_0$.
\end{corollary}
\begin{proof}
It is not restrictive to assume that $x_0=0_X$.
$\Rightarrow$ As $(-\cC, \,\overline{\mbox{cone}}(A+ \cC))$ has SSP, then  $(\cC, \, -(X\setminus \overline{\mbox{cone}}(A+ \cC))_0)$ has SSP too. Therefore, Remark \ref{remSSPcomplementario} yields that   $(-\cC, \,(X\setminus \overline{\mbox{cone}}(A+ \cC))_0)$ has SSP as well.  Now, apply Theorem \ref{thm:Idenfi_proper_con_QMIn}   (iii) and Corollary~\ref{coro:scalarization_proper_minimals}~(iii).  $\Leftarrow$ Fix $(f,\alpha) \in \cC^{a\#}_+$ such that $\underset{x \in A}{\min} \ f(x)+ \alpha\Vert x \Vert$ is attained only at $0_X$. 
By \cite[Theorem 1]{Gasimov2001}, $\cero \in \mbox{Be}(A,\cC)$.
\end{proof}

 \cite[Theorem~5.8]{Kasimbeyli2010} also characterizes Benson proper efficient points via (\ref{eq:scalarization_Q_proper_minimals}),  but under more restrictive assumptions than Corollary \ref{coro:Caract_Benson_Scalar_Problem} and, in addition, applying the separation property to  a sequence of $\epsilon$-conic neighbourhoods instead of to an only cone $- \overline{\mbox{cone}}(A-x_0+ \cC)$.

In the following, we study sufficient conditions to have $\mbox{GHe}(A,\cC)=\mbox{Be}(A,\cC)$. Such equality will lead to a characterization for Henig global proper efficient points via (\ref{eq:scalarization_Q_proper_minimals}). The set GHe($A,\cC$)  is contained in the set Be($A,\cC$) whenever $\cC$ is a closed, convex, and pointed cone (see \cite{Guerraggio1994}). The next result establishes the equality of such sets under some extra assumptions.  
\begin{theorem}\label{thm:Be=He}
Let $X$ be a partially ordered normed space, $\cC$ the ordering cone, and $A\subset X$ a subset such that $A+\cC$ is convex. If $\cC$ has a weakly compact base, then $\mbox{GHe}(A,\cC)=\mbox{Be}(A,\cC)$.
\end{theorem}
\begin{proof}
$\subset$ is provided by \cite[Proposition 2.4.6 (i)]{Zalinescu2015} for  separated topological vector spaces and non closed cones.
$\supset$ Fix an arbitrary $x_0 \in$ Be($A,\cC$). It is not restrictive to assume that $x_0=0_X$. Then $0_X\in \mbox{Min}(A,\cC)$ and $\cero\in \mbox{Min}(\overline{\mbox{cone}}(A+\cC),\cC)$. Hence $(-\cC)\cap \overline{\mbox{cone}}(A+\cC)=\{\cero\}$.  On the other hand,  $A+\cC$ is convex. Then $\mbox{cone}(A+\cC)$ is convex,  implying that $\overline{\mbox{cone}}(A+\cC)$ is weak closed.  Now,  \cite[Theorem~5.2]{Kasimbeyli2010} applies and there exists a convex cone $\cK$ such that $-\cC\setminus \{\cero\}\subset \mbox{int}(\cK)$ and $\cK\cap\overline{\mbox{cone}}(A+\cC)=\{\cero\}$. Then $0_X \in \mbox{GHe}(A,\cC)$.
\end{proof}
The following result is a direct consequence of  Corollary \ref{coro:Caract_Benson_Scalar_Problem} and Theorem~\ref{thm:Be=He}.

\begin{corollary}\label{coro:charact_scalarization_GHe}
Let $X$ be a partially ordered normed space, $\cC$ the ordering cone,  and $x_0 \in A\subset X$ such that $A+\cC$ is convex. Assume that $\cC$  has a weakly compact base. If $(-\cC, \, \overline{\mbox{cone}}(A-x_0+ \cC))$ has SSP, then $x_0 \in$ GHe($A,\cC$) if and only if there exists $(f,\alpha)\in \cC^{a\#}_+$ such that $\min_{x\in A}\{f(x-x_0)+\alpha \| x-x_0 \|\}$ is attained only at $x_0$.
\end{corollary}

Corollary \ref{coro:Caract_Benson_Scalar_Problem} provides a sufficient condition to find elements in $\mbox{TBo}(A,\cC)$.  In the following, we will establish that it becomes a characterization if we assume the following geometric condition on $x_0 \in A$.  It is said that a set $A\subset X$ is starshaped at some $x_0 \in A$, if $\lambda x+(1-\lambda)x_0 \in A$ for every $x \in A$ and $\lambda \in [0,1]$. Since \cite[Corollary 3.46]{Jahn2004} establishes that $T(A,x_0)=\overline{\mbox{cone}}(A-x_0)$ whenever $A$ is starshaped at $x_0 \in A$, it follows the equivalence $x_0\in \mbox{Be}(A,\cC) \Leftrightarrow x_0\in \mbox{TBo}(A,\cC)$ whenever $A+\cC$ is starshaped at $x_0 \in A$. Consequently, we obtain the following characterization for tangentially Borwein proper efficient points.

\begin{corollary}\label{coro:charact_scalarization_TBo}
Let $X$ be a partially ordered normed space, $\cC$ the ordering cone, and $x_0 \in A\subset X$. Assume that $A+\cC$ is starshaped at $x_0$ and that $\cC$ has a weakly compact base. If $(-\cC, \,  \overline{\mbox{cone}}(A-x_0+ \cC))$ has SSP, then $x_0 \in$ TBo($A,\cC$) if and only if there exists $(f,\alpha)\in \cC^{a\#}_+$ such that $\min_{x\in A}\{f(x-x_0)+\alpha \| x-x_0 \|\}$ is attained only at $x_0$.
\end{corollary}

We finish this section with the following problem for future research.
\begin{problem}\label{problem}
 Is it possible to characterize $\cQ$-minimal points via SSP assuming any extra conditions?
\end{problem}
In Theorem \ref{thm:scalarization_Q_minimals} we provide necessary conditions for $\cQ$-minimal points.   On the other hand,  in Corollaries \ref{coro:Caract_Benson_Scalar_Problem}, \ref{coro:charact_scalarization_GHe},  and \ref{coro:charact_scalarization_TBo}, we  establish characterizations of Benson, global Henig, and tangentially Borwein proper efficient points, respectively, answering the former problem for such particular kinds of $\cQ$-minimal points.  So, it is of interest to solve Problem~\ref{problem} for any of the other types of $\cQ$-minimal points.

\section{Scalarization for approximate  proper  efficient points}\label{sec:approximate_proper}
In this section, we obtain optimal conditions  through scalarization for approximate proper efficient points in the senses of Benson and Henig. We obtain our results after extending the approach for Benson and Henig proper efficient points in the precedent section. 

Let us introduce the  terminology of approximate proper efficiency. From now on, the ordering cone $\cC \subset X$  is assumed to be closed, convex, and pointed. The notions of approximate efficiency are defined replacing the ordering cone $\cC$ by a non-empty set $D$ that approximates it. For a non-empty set $D\subset X\setminus\{\cero\}$, we define the set $D(\epsilon):=\epsilon D$,  for $\epsilon>0$,  and $D(0):=\mbox{cone}(D)\setminus \{\cero\}$. 
We also  introduce the family of sets
$\overline{\cH}:=\{\emptyset\not = D\subset X\setminus \{\cero\}\colon \overline{\mbox{cone}}(D)\cap (-\cC)=\{\cero\}\}$.  Notice that $D(\epsilon)\in \overline{\cH}$, for every $D\in \overline{\cH}$ and $\epsilon\geq 0$. Now, fixed any $D\in \overline{\cH}$, we  introduce the family $\cG(D):=\{\cC'\subset X\colon \cC' \mbox{ is an open convex cone, } \cC\setminus \{\cero\}\subset \cC',\, D\cap (-\cC')=\emptyset\}$. Note that $\cG(D(\epsilon))=\cG(D)$ for every $\epsilon\geq 0$.
The following notion was introduced by Gutierrez, Huerga, and Novo in \cite{Gutierrez2012} for locally convex spaces. 
\begin{definition}\label{defi:Benson_approximate}
Let $X$ be a partially ordered normed space, $\cC$ the ordering cone,  $A\subset X$ a subset, $\epsilon\geq 0$, and $D \in \overline{\cH}$.  We say that $x_0 \in A$ is a Benson $(D,\epsilon)$-efficient  point  of $A$,  written $x_0 \in \mbox{Be}(A, \cC,D, \epsilon)$, if $\cero \in \mbox{Min}(\overline{\mbox{cone}}(A+D(\epsilon)-x_0),\cC)$.
\end{definition}

For the notion of approximate Henig efficiency we take the characterization \cite[Theorem~3.3~(c)]{Gutierrez2016} (adapted to normed spaces) as a definition instead of the original \cite[Definition 3.1]{Gutierrez2016}.
\begin{definition}\label{defi:Henig_approximate}
Let $X$ be a partially ordered normed space, $\cC$ the ordering cone,  $A\subset X$ a subset, $\epsilon\geq 0$, and $D \in \overline{\cH}$. We say that $x_0 \in A$ is a Henig $(D,\epsilon)$-efficient  point  of $A$,  written $x_0 \in \mbox{He}(A, \cC,D, \epsilon)$, if there exists $\cC_{D,\epsilon}\in \cG(D)$  such that $\overline{\mbox{cone}}(A + D(\epsilon)-x_0) \cap (-\cC_{D,\epsilon})=\emptyset$.
\end{definition}

It is clear that $\mbox{He}(A, \cC,D, \epsilon)\subset \mbox{Be}(A, \cC,D, \epsilon)$.  We begin our analysis determining two necessary conditions for approximate proper efficiency in the sense of Benson. 
Following \cite{Gutierrez2006}, we denote by $\mbox{AMin}(g,A,\varepsilon)$ the set of $\varepsilon$-approximate solutions of the scalar optimization problem $\underset{x\in A}{\mbox{Min}} \,g(x)$, i.e.,  $\mbox{AMin}(g,A,\varepsilon)=\{x \in A\colon g(x)-\varepsilon \leq g(z), \ \forall z\in A\}$, where $g:X\rightarrow \mathbb{R}$,  $A\subset X$, $A\not = \emptyset$,  and $\varepsilon >0$.  By means of approximate solutions were derived necessary and sufficient conditions for $\epsilon$-efficient solutions in \cite{Gutierrez2006}.  On the other hand, for every $(f,\alpha) \in \cC^{a\#}_+$ and $x_0 \in A$, we denote by $g_{(f,\alpha,x_0)}$ the mapping defined by $g_{(f,\alpha,x_0)}(x)=f(x-x_0)+\alpha \| x-x_0 \|$ for every $x \in X$.  For simplicity of notation,  we write $g_{(f,\alpha)}$ instead of the sublinear map $g_{(f,\alpha,\cero)}$.
\begin{theorem} \label{thm:escalar_neces_approximate_Benson}
Let $X$ be a partially ordered normed space, $\cC$ the ordering cone,  $x_0 \in A\subset X$, $\epsilon\geq 0$, and $D \in \overline{\cH}$. 
Assume that  $(-\cC,\,\overline{\mbox{cone}}(A-x_0+D(\epsilon)))$ has SSP.  If $x_0 \in \mbox{Be}(A, \cC, D,\epsilon)$, then there exists $(f,\alpha) \in \cC^{a\#}_+$ such that:
\begin{itemize}
\item[(i)]  $\underset{x \in (A+D(\epsilon))\cup\{x_0\}}{\mbox{min}} \ f(x-x_0)+ \alpha\Vert x-x_0 \Vert$ is attained only at $x_0$.
\item[(ii)]  $x_0 \in \mbox{AMin}(g_{(f,\alpha,x_0)},A,\lambda)$ for $\lambda=  \underset{d \in D(\epsilon)}{\mbox{inf}} f(d)+ \alpha\Vert d  \Vert$. 
\end{itemize}
\end{theorem}
\begin{proof}
It is not restrictive to assume that $x_0=0_X$. Since $(-\cC,\,\overline{\mbox{cone}}(A+D(\epsilon)))$ has SSP, then $(\cC,\, -\overline{\mbox{cone}}(A+D(\epsilon)))$ has SSP too. Therefore, by Remark \ref{remSSPcomplementario},  $(\cC,\,-(X\setminus\overline{\mbox{cone}}(A+D(\epsilon))_0))$ has SSP as well.  As $\cero \in \mbox{Be}(A, \cC, D,\epsilon)$ implies $\cC\setminus \{\cero \} \subset -X\setminus\overline{\mbox{cone}}(A+D(\epsilon))$,  Theorem~\ref{thm:separacion_para_aplicar_Q_minimales} applies and there exists $(f,\alpha) \in \cC^{a\#}_+$ such that $f(c)+ \alpha\Vert c \Vert <0<f(x)+\alpha \Vert x \Vert$ for every $ c\in -\cC$, $c\not = \cero$, and  $x\in \overline{\mbox{cone}}(A+D(\epsilon))$, $x \not = \cero$.  Since $A+D(\epsilon) \subset \overline{\mbox{cone}}(A+D(\epsilon))$, it follows that  $0<f(x)+\alpha \Vert x \Vert$ for every $x \in A+D(\epsilon) $, $x \neq \cero$. Then, we have (i).  Let us prove (ii).  Since $g_{(f,\alpha)}$ is a sublinear map, we have
 $0<f(x)+\alpha \Vert x \Vert + f(d)+\alpha \Vert d \Vert $ 
 for every $x \in A$, $x\not =\cero$, and $d \in D(\varepsilon)$. Fixing $x=\cero$, we get that $\lambda:=\underset{d \in D(\epsilon)}{\mbox{inf}} f(d )+ \alpha\Vert d  \Vert  \geq 0$.  Consequently,  $0\leq f(x)+\alpha \Vert x \Vert + \lambda $  for every $x \in A$,  i.e.,  $g_{(f,\alpha,\cero)}(\cero) -\lambda\leq g_{(f,\alpha,\cero)}(x)$ for every $x \in A$. Then $\cero \in \mbox{AMin}(g_{(f,\alpha,x_0)},A,\lambda)$.
\end{proof}

Let us recall that a function $g:X\rightarrow \R$ is strongly monotonically increasing if for each $x$, $y\in X$, $y-x\in \cC\setminus \{\cero\} \Rightarrow g(x)<g(y)$.   It is clear that $g_{(f,\alpha)}$ is strongly monotonically increasing  for every $(f,\alpha) \in \cC^{a\#}_+$.  Monotonicity will be used in the proof of the following result showing that the necessary condition (i) in Theorem \ref{thm:escalar_neces_approximate_Benson}  is also sufficient.
\begin{theorem} \label{thm:escalar_sufficient_approximate_Benson}
Let $X$ be a partially ordered normed space, $\cC$ the ordering cone,  $x_0 \in A\subset X$, $\epsilon\geq 0$, and $D \in \overline{\cH}$.  If there exists $(f,\alpha) \in \cC^{a\#}_+$ such that $f(x-x_0)+ \alpha\Vert x-x_0 \Vert\geq 0$ for every $x \in A+D(\epsilon)$, then $x_0 \in \mbox{Be}(A, \cC, D,\epsilon)$.
\end{theorem}
\begin{proof}
This proof is an adaptation of \cite[Theorem 1]{Gasimov2001}. It is not restrictive to assume that $x_0=0_X$.  Fix $(f,\alpha) \in \cC^{a\#}_+$ such that $f(x)+ \alpha\Vert x \Vert\geq 0$ for every $x \in A+D(\epsilon)$. Then $f(y)+ \alpha\Vert y \Vert\geq 0$ for every $y \in A+D(\epsilon)$.   We will show that $\cero \in \mbox{Be}(A,\cC,D,\epsilon)$.  Clearly, $f(y)+ \alpha\Vert y \Vert\geq 0$ for every $y \in {\mbox{cone}}(A+D(\epsilon))$ and, by continuity,  $f(y)+ \alpha\Vert y \Vert\geq 0$ for every $y \in \overline{\mbox{cone}}(A+D(\epsilon))$. Now, assume that  $\cero \not \in \mbox{Be}(A,\cC,D,\epsilon)$. Then there exists $\bar{y}\in \overline{\mbox{cone}}(A+D(\epsilon))\cap (-\cC\setminus \{\cero\})$. But strongly monotonicity of $g_{(f,\alpha)}$ (\cite[Theorem 3.5]{Kasimbeyli2010}) implies that $f(\bar{y})+\alpha \| \bar{y} \|<0$, a contradiction. 
\end{proof}

As a consequence of the former result and Theorem \ref{thm:escalar_neces_approximate_Benson}  (i) we obtain the  following characterization for  approximate proper efficiency in the sense of Benson.
\begin{corollary} \label{cor:escalar_characterization_approximate_Benson}
Let $X$ be a partially ordered normed space, $\cC$ the ordering cone,  $x_0 \in A\subset X$, $\epsilon\geq 0$, and  $D \in \overline{\cH}$. Assume that $(-\cC,\,\overline{\mbox{cone}}(A-x_0+D(\epsilon)))$ has SSP. Then $x_0 \in \mbox{Be}(A, \cC, D,\epsilon)$ if and only if there exists $(f,\alpha) \in \cC^{a\#}_+$ such that $\underset{x \in (A+D(\epsilon))\cup\{x_0\}}{min} \ f(x-x_0)+ \alpha\Vert x-x_0 \Vert$ is attained only at $x_0$.
\end{corollary}
Unfortunately, the necessary condition (ii) in Theorem \ref{thm:escalar_neces_approximate_Benson}  is not sufficient, as the following example shows.
\begin{example}
Take $X=\mathbb{R}^2$, the norm $\| (x,y) \| = \sqrt{x^2+y^2}$, $ \cC=\mbox{cone}(\{(1,1)\})$,  $A=\{(x,0)\in \mathbb{R}^2:-1\leq x \leq 0 \}$, and $D=\{(x,y)\in \mathbb{R}^2: \| (x,y)\|=1, x\geq 0, y\leq 0 \}$. Fix $\varepsilon=1$, $x_0=(0,0)\in A$, $f=(1,1)\in X^*$, and  $\alpha=\frac{4}{3}>0$.  Then $(0,0)\in \mbox{AMin}(g_{(f,\alpha,x_0)},A,\lambda)$ but $(0,0)\not\in \mbox{Be}(A, \cC, D,\epsilon)$.
\end{example}
\begin{proof}
If $\lambda=  \underset{d \in D(\epsilon)}{\mbox{inf}} f(d)+ \alpha\Vert d  \Vert =\frac{1}{3}$ and $g_{(f,\alpha,x_0)}(x,0)=x+\dfrac{4}{3} \mid x\mid=-\frac{1}{3}x \geq0> -\frac{1}{3}=-\lambda$,  for every $-1\leq x \leq 0 $, then $(0,0)\in \mbox{AMin}(g_{(f,\alpha,x_0)},A,\lambda)$.  However, $(-1,-1) \in \overline{\mbox{cone}}(A+D(\epsilon))$ because $(-1,-1)=(-1,0)+(0,-1)$, $(-1,0)\in A$, and $(0,-1)\in D(\epsilon)$. Therefore, $(-1,-1) \in (-\cC) \cap \overline{\mbox{cone}}(A+D(\epsilon))$, which implies that  $(0,0)\not\in \mbox{Be}(A, \cC, D,\epsilon)$.
\end{proof}
The preceding example leads us to the following natural question.
\begin{problem}\label{prob:escalarizacion_Benson_Approximate}
Is it possible to characterize approximate Benson proper efficient points via  $\epsilon$-approximate solutions assuming any extra conditions?
\end{problem}
We devote the rest of this section to study approximate Henig proper efficiency. An easy adaptation of the proof of Theorem \ref{thm:escalar_neces_approximate_Benson} gives the following necessary conditions for approximate proper solution in the sense of Henig. \\

\begin{theorem} \label{thm:escalar_nece_Henig_approximate}
Let $X$ be a partially ordered normed space, $\cC$ the ordering cone,  $A\subset X$ a subset, $\epsilon\geq 0$, and $D \in \overline{\cH}$.  Let $x_0 \in \mbox{He}(A, \cC,D, \epsilon)$ and the corresponding $\cC_{D,\epsilon}\in \cG(D)$.  If  $(-\cC_{D,\epsilon},\,\overline{\mbox{cone}}(A + D(\varepsilon)-x_0))$ has SSP, then there exists $(f,\alpha) \in {(\cC_{D,\epsilon})}^{a\#}_+ \subset \cC^{a\#}_+$ such that:
\begin{itemize}
\item[(i)] $\underset{x \in (A+D(\epsilon))\cup\{x_0\}}{\mbox{min}} \ f(x-x_0)+ \alpha\Vert x-x_0 \Vert$ is attained only at $x_0$.
\item[(ii)] $x_0 \in \mbox{AMin}(g_{(f,\alpha,x_0)},A,\lambda)$ for $\lambda=  \underset{d \in D(\epsilon)}{\mbox{inf}} f(d)+ \alpha\Vert d  \Vert$. 
\end{itemize}
\end{theorem}
Since $\mbox{He}(A, \cC,D, \epsilon)\subset \mbox{Be}(A, \cC,D, \epsilon)$, Theorem \ref{thm:escalar_neces_approximate_Benson} also provides  necessary conditions for Henig approximate proper solutions.  Furthermore,  when the former inclusion becomes a set equality, Corollary \ref{cor:escalar_characterization_approximate_Benson} provides a characterization  Henig approximate proper solutions.  This leads to the last results in the work.  Before stating them, we introduce the notion of approximating family of cones.
\begin{definition}\label{defi:approximating_cones}
Let $X$ be a partially ordered normed space and $\cC$ the ordering cone.
\begin{itemize}
\item[(i)] Let  $\cF=\{\cC_n\subset X\colon n \in \N\}$ be a family of decreasing (with respect to the inclusion)  solid, closed,  and pointed convex cones. We say that $\cF$ approximates $\cC$ if $\cC\setminus \{\cero\}\subset \mbox{int}(\cC_n)$ eventually (i.e., there exists $n_0 \in \N$ such that $\cC\setminus \{\cero\}\subset \mbox{int}(\cC_n)$ for every $n \geq n_0$) and $\cC=\cap_n \cC_n$.
\item[(ii)] Let $\cF$ be an approximating family  of cones for $\cC$. We say that $\cF$ separates $\cC$ from a closed cone $\cK\subset X$ if $\cC\cap \cK=\{\cero\} \Rightarrow \cC_n\cap \cK=\{\cero\} \mbox{ eventually}$.
\end{itemize}
\end{definition}
Given $D\subset X\setminus \{\cero\}$, $\epsilon>0$, and $x \in X$, we denote by $\cS(D(\epsilon),x)$ the set of all families of cones that approximate $\cC$ and separate $\cC$ from the cone $-\overline{\mbox{cone}}(A-x+D(\epsilon))$. 
\begin{corollary} \label{cor:escalar_characterization_approximate_Henig}
Let $X$ be a partially ordered normed space, $\cC$ the ordering cone,  $x_0\in A\subset X$, $\epsilon\geq 0$, and  $D \in \overline{\cH}$. Assume that $(-\cC,\,\overline{\mbox{cone}}(A-x_0+D(\epsilon)))$ has SSP and $\cS(D(\epsilon),x_0)\not = \emptyset$. Then $x_0 \in \mbox{He}(A, \cC, D,\epsilon)$ if and only if there exists $(f,\alpha) \in \cC^{a\#}_+$ such that $\underset{x \in (A+D(\epsilon))\cup\{x_0\}}{min} \ f(x-x_0)+ \alpha\Vert x-x_0 \Vert$ is attained only at $x_0$. 
\end{corollary}
\begin{proof}
By \cite[Theorem 3.1]{Gutierrez2019},  $x_0 \in \mbox{He}(A, \cC, D,\epsilon) \Leftrightarrow x_0 \in \mbox{Be}(A, \cC, D,\epsilon)$. Now,  Corollary \ref{cor:escalar_characterization_approximate_Benson} applies.
\end{proof}

Note that $\cS(D(\epsilon),x_0)\not=\emptyset$ whenever $X$ is finite-dimensional (\cite[Theorem~2.1]{Henig1982}) or if $\cC$ has a weakly compact base and $\overline{\mbox{cone}}(A-x_0+D(\epsilon))$ is weakly closed (\cite[Theorem~2.3]{Gutierrez2019}). Therefore, we have the following.
\begin{corollary} \label{cor:escalar_characterization_approximate_Henig_finito}
Let $X$ be a partially ordered normed space, $\cC$ the ordering cone,  $x_0\in A\subset X$, $\epsilon\geq 0$, and  $D \in \overline{\cH}$. Assume that $(-\cC,\,\overline{\mbox{cone}}(A-x_0+D(\epsilon)))$ has SSP and  at least one of the following assertions hold:
\begin{itemize}
\item[(i)] $X$ has finite dimension.
\item[(ii)] $\cC$ has a weakly compact base and $\overline{\mbox{cone}}(A-x_0+D(\epsilon))$ is weakly closed.
\end{itemize}
Then $x_0 \in \mbox{He}(A, \cC, D,\epsilon)$ if and only if there exists $(f,\alpha) \in \cC^{a\#}_+$ such that \[\underset{x \in (A+D(\epsilon))\cup\{x_0\}}{min} \ f(x-x_0)+ \alpha\Vert x-x_0 \Vert,\] is attained only at $x_0$.
\end{corollary}

\section{Conclusions}
\label{sec:conclusions}
In this work, we provide optimal sufficient conditions with a sublinear function for  Henig global proper efficient points, Henig proper efficient points, super efficient points, Benson proper efficient points, Hartley proper efficient points, Hurwicz proper efficient points, Borwein proper efficient points, and tangentially Borwein proper efficient points; in the case of Benson proper efficiency the optimal condition becomes a characterization. The approach is done in a unified way considering such proper efficient points as $\cQ$-minimal points. For every type of proper efficient point we apply a separation property to a fixed $\cQ$-dilation of the ordering cone.  For future research we ask if it is possible to characterize $\cQ$-minimal points in general via SSP assuming any extra conditions. In the last part of the work, we adapt our arguments to obtain new characterizations of Benson and Henig approximate proper efficient points  through scalarizations.  We also provide necessary conditions for approximate Benson proper efficient points via $\epsilon$-approximate solutions and we ask if it is possible to extend such a result to a characterization assuming any extra condition. Our results are established  in the setting of normed spaces and they do not impose any kind of convexity and boundedness assumption. 

\backmatter

\bmhead{Acknowledgments}
We thank the referees for their valuable suggestions.  In particular, for the clarification on the relationship between the sets $\cC^{\#}$ and $\mbox{int}(\cC^*)$.

\section*{Declarations}

\begin{itemize}
\item Fernando Garc\'ia-Casta\~no acknowledge and .M. A. Melguizo-Padial acknowledge the financial support from the Spanish Ministry of Science, Innovation and Universities (MCIN/AEI) under grant PID2021-122126NB-C32, co-funded by the European Regional Development Fund (ERDF) under the slogan "A way of making Europe".

\item Conflict of interest/Competing interests.  The authors have no competing interests to declare that are relevant to the content of this article.
\end{itemize}


\bibliography{references}


\end{document}